\documentclass[draft,reqno,11pt,oneside]{article} 

\usepackage{amssymb}
\usepackage{amsthm}
\usepackage{amssymb}
\usepackage{longtable} 
\usepackage{amsmath}
\usepackage{ifthen} 
\usepackage{upref}
\usepackage{bbold} 
\usepackage{latexsym} 
\usepackage[all]{xy}
\usepackage{fancyhdr} 
\usepackage{makeidx} 
\newfont{\chen}{cmex10 at 9pt} \newfont{\chee}{cmex10 at 10pt} 
\newfont{\cheu}{cmex10 at 11pt} 
\makeindex 
\begin{document} 
\newtheorem{lemma}{Lemma}[section]
\newtheorem*{eldemo}{Sketch of proof}
\newtheorem{satz}[lemma]{Theorem}
\newtheorem{prop}[lemma]{Proposition}
\newtheorem{defi}[lemma]{Definition}
\newtheorem{bei}[lemma]{Example} 
\newtheorem{verm}[lemma]{Conjecture} 
\newtheorem{kor}[lemma]{Corollary} 
\renewcommand{\proofname}{Proof} 
\newtheorem{bez}[lemma]{Notation} 
\newtheorem{bem}[lemma]{Remark}
\newtheorem{fall}[lemma]{Case}
\newtheorem*{example}{Example} 
\newtheorem{defsatz}[lemma]{Definition and Theorem} 
\newtheorem{lemmaap}{Lemma}[section]
\newtheorem{satzap}[lemmaap]{Theorem}
\newtheorem{propap}[lemmaap]{Proposition}
\newtheorem{defiap}[lemmaap]{Definition}
\newtheorem{beiap}[lemmaap]{Example}  
\newtheorem{korap}[lemmaap]{Corollary}  
\newtheorem{bezap}[lemmaap]{Notation} 
\newtheorem{bemap}[lemmaap]{Remark}

\def \Q{{\mathbb{Q}}}
\def \Ar{{\mathcal{A}}}
\def \G{{\mathcal{G}}} 
\def \pu{{\mbox{.}}} 
\def \ps{{\Phi^s}} 
\def \LUC{{\mathrm{LUC}}} 
\def \RUC{{\mathrm{RUC}(\G)}}
\def \RSp{{\widetilde{R}}^{(p)}} 
\def \RSq{{\widetilde{R}}^{(q)}} 
\def \wwp{{\mathfrak{P}}} 
\def \phii{{\overline{\Phi}}} 
\def \Mun{{\mathcal{M}}}
\def \Ce{{\mathrm{C}}}
\def \Cede{{\mathbb{C}}} 
\def \lpp{{L_p}} 
\def \Te{{\mathcal{T}}}
\def \sis{{\mathcal{S}}} 
\def \th{{\widetilde{\Theta}}} 
\def \mc{{\mathcal{MC}}} 
\def \ml{{\mathcal{MLUC}}} 
\def \sus{{\subseteq}} 
\def \mr{{\mathcal{MRUC}}} 
\def \N{{\mathbb{N}}}
\def \al{{\alpha}} 
\def \a{{\alpha}} 
\def \suppp{{\mathrm{supp}}} 
\def \cb{{\mathrm{cb}}} 
\def \A{{\mathcal{A}}}
\def \llu{{\mathrm{LUC}(\G)}} 
\def \B{{\mathcal{B}}}
\def \kapp{{\mathfrak{k}}} 
\def \ru{{{\RUC}^*}} 
\def \L{{\mathfrak{L}}}
\def \EL{{\mathcal{L}}}
\def \gip{{\Gamma_{l, p}}} 
\def \UC{{{\mathrm{UC}}(\G)}} 
\def \gup{{{\widetilde{\Gamma_{l, p}}}}} 
\def \gurp{{\Gamma_{r, p}}} 
\def \gurrp{{{\widetilde{\Gamma_{r, p}}}}} 
\def \M{{\rm{M}}}
\def \en{{\mathcal{N}}}
\def \gr{{\widetilde{\Gamma_r}}} 
\def \gre{{\Gamma_r}} 
\def \em{{\mathcal{M}}} 
\def \emli{{\mathcal{ML}_\infty}} 
\def \ccc{{\mathcal{C}}}
\def \Ha{{\mathfrak{H}}}
\def \l1{{L_1(\G)}} 
\def \li{{L_{\infty}(\G)}} 
\def \H{{\mathcal{H}}}
\def \Qu{{\mathfrak{Q}}}
\def \cbgl{~{\stackrel{\mathrm{cb}}{=}}~} 
\def \FS{{\mathcal{FS}}} 
\def \gammas{{\widetilde{\gamma}}} 
\def \C{{\mathcal{C}}}
\def \mr{{\mathcal{MRUC}}}
\def \Be{{\mathfrak{B}}}
\def \bar{{~|~}} 
\def \Ka{{\mathcal{K}}}
\def \rr{{\mathcal{R}}} 
\def \ceb{{\C\B(\B(L_2(\G)))}} 
\def \id{{{\mathrm{id}}}} 
\def \cep{{\C\B(\B(L_p(\G)))}} 
\def \ep{{\B(\B(L_p(\G)))}}
\def \Is{{\widetilde{I}}} 
\def \conv{{\mathrm{conv}}} 
\def \convs{{\widetilde{\mathrm{conv}}}} 
\def \convsp{{\widetilde{\mathrm{conv_p}}}} 
\def \es{{\B(\B(L_2(\G)))}}
\def \go{{\gip}} 
\def \tt{{\overline{\otimes}}} 
\def \ti{{\stackrel{\vee}{\otimes}}} 
\def \oo{{{\overline{\otimes}}}} 
\def \tp{{\widehat{\otimes}}} 
\def \otp{{\otimes}} 
\def \zet{{Z_t(\lu)}} 
\makeatletter 
\newcommand{\essu}
{\mathop{\operator@font ess\mbox{-}sup}} 
\newcommand{\essi}
{\mathop{\operator@font ess\mbox{-}inf}} 
\newcommand{\lisu}
{\mathop{\operator@font lim\,ess\mbox{-}sup}} 
\newcommand{\liin}
{\mathop{\operator@font lim\,ess\mbox{-}inf}} 
\makeatother 
\def \ga{{\Gamma_p}}
\def \gam{{\gamma_p}}
\def \gaw{{\widetilde{\ga}}}
\def \lu{{\LUC(\G)^*}}
\def \sa{{\overline{\Gamma_p(\M(\G))}^{w*}}}
\def \Ball{{\mathrm{Ball}}} 
\def \gurr{{\widetilde{\Gamma_r}}} 
\def \gur{{\Gamma_r}} 
\def \Ker{{\it{KERN}}}
\def \F{{\mathfrak{F}}}
\def \ad{{\mathrm{ad}}} 
\def \Ceee{{\mathbb{C}}}
\def \Ag{{\Ar \otimes_\gamma \Ar}}  
\def \Cee{{\mathrm{C}}}
\def \imink{{\iota_{\mathrm{minK}}}} 
\def \imin{{\iota_{\mathrm{min}}}} 
\def \Bild{{\it{BILD}}}
\def \Sz{{\mathcal{S}}}
\def \The{{\mathcal{T}}}
\def \cebe{{\C\B(\B(\Sz_2))}}
\def \cebes{{\C\B(\B(\Sz_2(L_2(G))))}} 
\def \pr{{\Sz_2 \otimes_h \B(\Sz_2) \otimes_h \Sz_2}} 
\def \su{{\rm{sup}}_{t \in \G}}
\def \mi{{\Gamma_2(\M(\G))}}
\def \Cb{{\mathcal{C}}}
\newcommand{\ens}{{\mathcal{N}}(L_p(\G))}
\newcommand{\Tee}{{\mathcal{T}}(\mathcal{H})}
\newcommand{\schutz}[1]{#1} 
\def \ma{{\Gamma_2(\lu)}}
\def \Hi{{\mathcal{H}}}
\def \bs{{\B^{\sigma}(\B(\Ha))}}
\def \bsa{{\C\B(\B(\H))}}
\def \bsi{{\B^s(\B(\Ha))}}
\def \gu{{\widetilde{\Gamma_l}}}
\def \Ge{{\mathcal{G}}} 
\def \EF{{\widetilde{F}}} 
\def \gi{{\Gamma_l}} 
\def \als{{\widetilde{\alpha}}} 
\def \betas{{\widetilde{\beta}}} 
\def \gir{{\Gamma_r}} 
\def \Se{{\mathcal{S}}} 
\def \muss{{\widetilde{\mu}}} 
\def \ge{{\Gamma}}
\def \gem{{\Gamma(\M(\G))}}
\def \MM{{\mathbf{M}}} 
\def \un{{\ell_{\infty}^*(\G)}}
\def \R{{\mathcal{R}}}
\def \cebr{{\C\B_{\R(\G)}(\B(L_2(\G)))}}
\def \Chi{{\chi}} 
\def \otg{{\otimes_\gamma}}
\def \cebre{{\C\B_{\R(\G)}(\B(\ell_2(\G)))}}
\def \cebrs{{\C\B_{\R(\G)^{'}}^{\sigma}(\B(L_2(\G)))}}
\def \cebra{{\C\B_{\R(\G)}^{\sigma}(\B(\ell_2(\G)))}}
\def \U{{\mathfrak{U}}}
\def \K{\Ka} 
\def \cebru{{\C\B_R(\B(L_2(\G)))}}
\def \lin{{\mathrm{lin}}} 
\def \cebri{{\C\B_{R}^{\sigma}(\B(\H))}}
\def \Tee{{\The(\H)}}
\def \gurrn{{\gurr^0}} 
\def \gin{{\gi^0}} 
\def \gun{{\gu^0}} 
\def \wap{{\mathrm{WAP}}} 
\def \ap{{\mathrm{AP}}} 
\def \mg{{\mathrm{M}(\G)}} 
\def \des{{\sqrt{x_j} \otimes M_{T_{\widecheck{b}}} \otimes \sqrt{x_j}}} 
\def \res{{(e_j)^{\frac{1}{p}} \otimes T_{\widecheck{b}} \otimes (e_j)^{\frac{1}{q}}}} 
\def \rem{{(e_j)^{\frac{1}{2}} \otimes M_{\widecheck{b}} \otimes (e_j)^{\frac{1}{2}}}} %

\title{Geometry of $C^*$-algebras, the bidual of their projective tensor product, and completely bounded module maps}
\author{Matthias Neufang} 
\date{} 
\maketitle 
\begin{abstract} 
\noindent 
Let $\Ar$ be a $C^*$-algebra, and consider the Banach algebra $\Ar \otimes_\gamma \Ar$, where $\otimes_\gamma$ 
denotes the projective Banach space tensor product; if $\Ar$ is commutative, this is the Varopoulos algebra $V_\Ar$. It has been an open problem for more than 35 years to determine 
precisely when $\Ar \otimes_\gamma \Ar$ is Arens regular; cf.\ \cite{ljes}, \cite{u1}, \cite{u2}. Even the situation for commutative $\Ar$, in particular the case $\Ar = \ell_\infty$, has remained unsolved. 
We solve this classical question for arbitrary $C^*$-algebras by using von Neumann algebra and operator space methods, mainly relying on versions of the (commutative and non-commutative) 
Grothendieck Theorem, and the structure of 
completely bounded module maps. Establishing these links allows us to show 
that $\Ar \otimes_\gamma \Ar$ is Arens regular if and only if $\Ar$ has the Phillips property; equivalently, $\Ar$ is scattered and has the Dunford--Pettis Property. 
A further equivalent condition is that $\Ar^*$ has the Schur property, or, again equivalently, 
the enveloping von Neumann algebra $\Ar^{**}$ is finite atomic, i.e., a direct sum of matrix algebras. 
Hence, Arens regularity of $\Ar \otimes_\gamma \Ar$ is encoded in the geometry of the $C^*$-algebra $\Ar$. 
In case $\Ar$ is a von Neumann algebra, we conclude that $\Ar \otimes_\gamma \Ar$ 
is Arens regular (if and) only if $\Ar$ is finite-dimensional. We also show that this does not generalize to the class of non-selfadjoint dual (even commutative) operator algebras. 
Specializing to commutative $C^*$-algebras $\Ar$, we obtain that $V_\Ar$ is Arens regular if and only if $\Ar$ is scattered. In fact, we determine precisely the centre of the bidual, 
namely, $Z({V_\Ar}^{**})$ is Banach algebra isomorphic to 
$\Ar^{**} \otimes_{eh} \Ar^{**}$, where $\otimes_{eh}$ denotes the extended 
Haagerup tensor product. We deduce that $V_\Ar$ is strongly Arens irregular (if and) only if $\Ar$ is finite-dimensional. Hence, $V_\Ar$ is neither Arens regular nor 
strongly Arens irregular, if and only if $\Ar$ is non-scattered (as mentioned above, this is new even for the case 
$\Ar = \ell_\infty$). 
\end{abstract}

\section{Introduction} \label{intro}
Let $\Ar$ be a $C^*$-algebra. 
We shall be concerned with the Banach algebra $\Ar \otimes_\gamma \Ar$, where we write $\otimes_\gamma$ 
for the projective Banach space tensor product. In case $\Ar$ is commutative, this is the so-called Varopoulos algebra $V_\Ar$. As is well-known, 
any $C^*$-algebra $\Ar$ is Arens regular, i.e., the left and right Arens product on the bidual $\Ar^{**}$ coincide; 
$\Ar^{**}$ with this product is the enveloping von Neumann algebra of $\Ar$. The problem of characterizing 
those $C^*$-algebras $\Ar$ such that $\Ar \otg \Ar$ is Arens regular, has to our knowledge first been systematically studied in the Ph.D.\ thesis \cite{ljes} of M.\ Ljeskovac, 
written in 1981 under the supervision of A.M.\ Sinclair (Edinburgh); apparently independent from this work are the articles 
\cite{u1} and \cite{u2} by A.\ \"Ulger, and \cite{la-ul} by A.T.-M.\ Lau and A.\ \"Ulger. Results regarding Arens regularity of 
$\Ar \otg \Ar$ have been obtained by these authors in several cases. However, even the case of general commutative $C^*$-algebras has thus far not been settled: it is known 
that $c_0 \otg c_0$ is Arens regular while $C(\Ge) \otg C(\Ge)$ is not, where $\Ge$ is an infinite compact group; but the case of $\ell_\infty \otg \ell_\infty$, for instance, has been open. 
In summary, no characterization of commutative $C^*$-algebras 
with Arens regular projective tensor square is known, let alone the situation for non-commutative $C^*$-algebras. 
In this paper, we shall use von Neumann algebra and operator space techniques to solve this question for arbitrary $C^*$-algebras, and present further results on the 
structure of the bidual. 
\par 
Our approach may be described in essence as follows. Our starting point is to pass, based on (commutative and non-commutative) 
versions of the Grothendieck Theorem, from the projective Banach space tensor product to the Haagerup tensor product, and then to use the latter and the extended Haagerup tensor product, 
as well as their intimate link to completely bounded module maps on $B(\H)$, where $\H$ is a Hilbert space: in this fashion, the left Arens product translates into composition of operators on $B(\H)$, 
and results on automatic normality 
of these maps yield structural results on the topological centre of the bidual algebra in question. By establishing these connections, we are then able to completely characterize Arens regularity 
of the projective tensor square of an arbitrary $C^*$-algebra in terms of intrinsic properties of the latter. As we shall prove, Arens regularity of 
$\Ar \otg \Ar$ 
is encoded in the geometry of the $C^*$-algebra $\Ar$. 
\par 
In section \ref{intro} we provide some preliminary results pertaining to $C^*$- and operator space theory which will be needed in the sequel. In section \ref{comme}, 
for commutative $C^*$-algebras $\Ar$, we determine precisely the centre of the bidual of the Varopoulos algebra $V_{\Ar}$, 
showing that it is Banach algebra isomorphic to $\Ar^{**} \otimes_{eh} \Ar^{**}$, 
where $\otimes_{eh}$ denotes the extended Haagerup tensor product. It follows that $V_\Ar$ is strongly Arens irregular (if and) only if $\Ar$ is finite-dimensional. 
\par 
Section \ref{general} gives our main result, stating that $\Ar \otimes_\gamma \Ar$ is Arens regular if and only if $\Ar$ has the Phillips property; equivalently, $\Ar$ is scattered and has the 
Dunford--Pettis Property. Further equivalent statements are that $\Ar^*$ has the Schur property, or, again equivalently, 
the enveloping von Neumann algebra $\Ar^{**}$ is finite atomic, i.e., a direct sum of matrix algebras. In fact, we prove that if $\Ar$ has the Phillips property, then 
$\Ar \otimes_\gamma \mathcal{B}$ is Arens regular for any $C^*$-algebra $\mathcal{B}$. Moreover, we deduce that for a von Neumann algebra $\Ar$, Arens regularity of $\Ar \otg \Ar$ is 
equivalent to $\Ar$ being finite-dimensional. We also show that this characterization does not extend, in general, 
to non-selfadjoint dual operator algebras, even commutative ones: indeed, given any $p \in [1, \infty )$, the algebra 
$\ell_p$ ($O \ell_p$) with pointwise product is (completely) isomorphic and $w^*$-homeomorphic 
to a dual operator algebra with Arens regular projective tensor square; here, $O \ell_p$ denotes Pisier's operator space 
structure on $\ell_p$, obtained by complex interpolation. 
\par 
Considering commutative $C^*$-algebras $\Ar$, we obtain that $V_\Ar$ is Arens regular if and only if $\Ar$ is scattered. Combined with our earlier result, this 
shows that the Varopoulos algebra $V_\Ar$ is neither Arens regular nor strongly Arens irregular (i.e., the centre $Z({V_\Ar}^{**})$ lies strictly between 
$V_\Ar$ and its bidual), if and only if $\Ar$ is non-scattered. 
This result is already new for the fundamental example $\ell_\infty$. Indeed, already the question whether $\ell_\infty \otg \ell_\infty$ is Arens regular has been open since 1981. 

\section{Preliminaries} \label{prelim} 
We start by recalling the definitions of geometric properties for a $C^*$-algebra $\Ar$ and a Banach space $E$ which we shall encounter. 
\begin{itemize} 
\item $\Ar$ is scattered if any positive functional on $\Ar$ is a finite or countable (pointwise) sum of pure functionals (cf.\ \cite{jen}). 
\item $E$ has the Dunford--Pettis property (DPP) if any weakly compact operator from $E$ into any Banach space is completely continuous; equivalently, 
for any weakly null sequences $x_n$ in $E$ and $f_n$ in $E^*$, the sequence $\langle f_n , x_n \rangle$ tends to $0$ (going back to Grothendieck's classical work \cite{gro}). 
\item $E$ has the Phillips property if the Dixmier projection $E^{***} \to E^*$ is sequentially $w^*$--norm continuous (cf.\ \cite{fu}). 
\item $E$ has the Schur property if in $E$ any weakly convergent sequence is norm convergent. 
\end{itemize} 
We now review some related structure results for $C^*$-algebras; see \cite[Corollary VII-10]{ggms}, \cite[Theorem 3]{chu} and \cite[Theorem 2.2]{jen} for the first, 
and \cite[Lemma 3.1]{fu} together with \cite[Theorems 3.4 and 3.6]{la-ul} for the second. 
\begin{satz} \label{gen1} 
Let $\Ar$ be a $C^*$-algebra. Then the following are equivalent: 
\begin{itemize} 
\item[(i)] $\Ar$ is scattered; 
\item[(ii)] $\Ar$ does not contain an isomorphic copy of $\ell_1$; 
\item[(iii)] $\Ar^*$ has the Radon--Nikodym Property (RNP); 
\item[(iv)] $\Ar^{**}$ is atomic, i.e., a direct sum of $B(\H_i)$ for Hilbert spaces $\H_i$. 
\end{itemize} 
\end{satz} 
\begin{satz} \label{gen2} 
Let $\Ar$ be a $C^*$-algebra. Then the following are equivalent: 
\begin{itemize} 
\item[(i)] $\Ar$ has the Phillips property; 
\item[(ii)] $\Ar$ is scattered and has the DPP; 
\item[(iii)] $\Ar^*$ has the Schur property; 
\item[(iv)] $\Ar^{**}$ is finite atomic, i.e., a direct sum of matrix algebras. 
\end{itemize} 
\end{satz} 
We now fix some terminology. By ``$w^*$-continuous" we mean ``$w^*$-$w^*$-continuous". 
We write $\Box$ for the left Arens product in the bidual of a Banach algebra $\mathcal{A}$. Recall that, for $X , Y \in \mathcal{A}^{**}$, 
we have 
$$\langle X \Box Y , f \rangle := \langle X , Y \Box f \rangle \quad (f \in \mathcal{A}^*)$$
where 
$$\langle Y \Box f , a \rangle := \langle Y , f \Box a \rangle \quad (a \in \mathcal{A});$$ 
of course, $\langle f \Box a , b \rangle := \langle f , ab \rangle$ for all $b \in \mathcal{A}$. Note that, if $x_i$ and $y_j$ are nets in $\mathcal{A}$ converging 
$w^*$ to $X$ and $Y$, respectively, we have 
$$X \Box Y = \text{$w^*$-}\lim_i \text{$w^*$-}\lim_j x_i y_j .$$ 
Obviously, the map $\mathcal{A}^{**} \ni X \mapsto X \Box Y$ is $w^*$-continuous. The set of all $X \in \mathcal{A}^{**}$ such that the map 
$\mathcal{A}^{**} \ni Y \mapsto X \Box Y$ is $w^*$-continuous, is called the (left) topological centre, denoted by $Z_t(\mathcal{A}^{**})$. 
Now, $\Ar$ is said to be Arens regular if $Z_t(\mathcal{A}^{**})$ is maximal, i.e., equals $\Ar^{**}$ (this is equivalent to saying that the left and right Arens 
products on $\Ar^{**}$ coincide, but we shall only use the left one); also, $\Ar$ is said to be (left) strongly Arens irregular if $Z_t(\mathcal{A}^{**})$ is minimal, i.e., equals $\Ar$. 
Topological centres have been studied intensely in abstract harmonic analysis in recent years; see, e.g., the articles \cite{dl}, \cite{dls}, \cite{mmm}, \cite{hnr0}, \cite{la-lo2}, 
\cite{blms} and \cite{losert}. Let us mention that, 
for any locally compact group $\Ge$, the group algebra $L_1(\Ge)$ and the measure algebra $M(\Ge)$ are strongly Arens irregular; 
these are the main theorems of \cite{la-lo} and \cite{losert}, respectively, and we shall make use of the former result in section \ref{general}. 
\par 
We write $Z(\mathcal{A})$ for the centre of an algebra $\mathcal{A}$. Note that if $\Ar$ is a commutative Banach algebra, 
$Z_t(\mathcal{A}^{**})$ equals the centre $Z(\Ar^{**})$. Even in this situation, determining $Z_t(\mathcal{A}^{**})$ can be very difficult. For instance, denoting by $A(\Ge)$ 
the Fourier algebra of a locally compact group $\Ge$, the centre $Z(A(\Ge)^{**})$ is not known in general, even when $\Ge$ is discrete or compact. Indeed, it is open if 
$A(\Ge)$ can be Arens regular for discrete groups without infinite amenable subgroups nor copies of free groups, such as Olshanskii's `Tarski monsters'; cf.\ \cite{forrest}, \cite{blms}. 
In the case of infinite compact groups, it is known that $A(\Ge)$ is not Arens regular; cf.\ \cite{forrest}. As the projective tensor product $\Ar := C(\Ge) \otimes_\gamma C(\Ge)$ contains 
a subalgebra isomorphic to $A(\Ge)$, by Varopoulos's classical work, one thus obtains that $\Ar$ is not Arens regular, as mentioned in section \ref{intro}; cf.\ \cite{u2}. 
However, determining when $A(\Ge)$ is strongly Arens irregular, is more complicated: for instance, 
$A(SU(3))$ is not, but $A(SU(3)^{\aleph_0})$ and $A(SU(3)^{\aleph_1})$ are; cf.\ \cite{unpub}, \cite{la-lo2}, \cite{mmm}. 
\par 
Given a set $S \in B(\H)$, where $\H$ is a Hilbert space, we write $S'$ for its commutant. 
Let $\Ar \subseteq B(\H)$ and $\B \subseteq B(\Ka)$ be $C^*$-algebras. We denote by $\Ar^{op}$ the opposite algebra. 
We denote by $CB_{\Ar , \B} (B(\Ka , \H))$ the algebra of completely bounded operators on $B(\Ka , \H)$ that 
are left $\Ar$- and right $\B$-module maps, i.e., $\Phi (aTb) = a \Phi(T) b$ for all $T \in B(\Ka , \H)$ and $a \in \Ar$, $b \in \B$. If $\H = \Ka$ and $\Ar = \B$, we simply write $CB_{\Ar} (B(\H))$. 
We denote by $CB^{\sigma}_{\Ar , \B} (B(\Ka , \H))$ the subalgebra of normal (i.e., $w^*$-continuous) such maps. 
Of particular importance in our approach is the work of Magajna \cite{mag}, which contains a further development of the results obtained in \cite{HW} by Hofmeier--Wittstock. 
\par 
We write $\otg$ for the projective Banach space tensor product, 
$\otimes_h$ for the Haagerup tensor product, $\otimes_{eh}$ for the 
extended Haagerup tensor product, and $\otimes_{\sigma h}$ for the normal Haagerup tensor product. For the convenience of the reader, we shall briefly recall the definitions and some 
fundamental relations. For Banach spaces $E$ and $F$, the projective tensor product $E \otg F$ is the completion of the algebraic tensor product $E \otimes F$ with respect to 
the norm 
$$\| u \|_\gamma = \inf \left \{  \sum_{k=1}^m \| x_k \| \| y_k \| \mid u = \sum_{k=1}^m x_k \otimes y_k \right \} .$$ 
Now let $E$ and $F$ be operator spaces, $n \in \mathbb{N}$, and $u = [u_{ij}] \in M_n(E \otimes F)$. 
The Haagerup norm of $u$ is defined by $\| u \|_h = \inf \{ \| x \| \| y \| \}$, 
with the infimum taken over all $p \in \mathbb{N}$, and all representations $u = x \odot y$, where $x \in M_{n,p} (E)$ and $y \in M_{p,n} (F)$; 
here, $x \odot y = [ \sum_{k=1}^p x_{ik} \otimes y_{kj} ]$. The Haagerup tensor product $E \otimes_h F$ is the completion of $E \otimes F$ 
with respect to the Haagerup matrix norms. For $C^*$-algebras $\Ar$ and $\B$, the Haagerup norm on $\Ar \otimes \B$ can be expressed as 
$$\| u \|_h = \inf \left \{ \left\| \sum_{k=1}^n a_k {a_k}^* \right\|^{\frac{1}{2}} \left\| \sum_{k=1}^n {b_k}^* b_k \right\|^{\frac{1}{2}} \mid u = \sum_{k=1}^n a_k \otimes b_k \right \} .$$ 
Note that on $E \otimes F$ we have $\| \cdot \|_h \leq \| \cdot \|_{\gamma}$. If $\Ar$ and $\B$ are commutative $C^*$-algebras, the classical Grothendieck Inequality says that 
$$\| \cdot \|_\gamma \leq K \| \cdot \|_h ,$$ 
where $K$ is the Grothendieck constant; cf., e.g., \cite{ks}. 
\par 
The extended Haagerup tensor product $E \otimes_{eh} F$ is defined as the subspace of $(E^* \otimes_h F^*)^*$ corresponding to the 
completely bounded bilinear forms $E^* \times F^* \to \mathbb{C}$ which are separately $w^*$-continuous. Note that the $w^*$-Haagerup tensor product of dual operator spaces is defined by 
$$E^* \otimes_{w^*h} F^* = (E \otimes_h F)^* ;$$ 
so we have $E^* \otimes_{eh} F^* = E^* \otimes_{w^*h} F^*$. Finally, the normal Haagerup tensor product of dual operator spaces is defined as 
$$E^* \otimes_{\sigma h} F^* = (E \otimes_{eh} F)^* .$$ 
Hence, we have $(E \otimes_h F)^{**} = (E^* \otimes_{eh} F^*)^* = E^{**} \otimes_{\sigma h} F^{**}$. 
More information on these tensor products can be found, e.g., in \cite{BL1}, \cite{ER-book}, \cite{pisier}, and \cite{ER}. 
\par 
As noted above, the classical Grothendieck Inequality yields a canonical isomorphism between $\Ar \otg \Ar$ and $\Ar \otimes_h \Ar$ for 
commutative $C^*$-algebras $\Ar$, which plays a significant role in our approach; as do the consequences of the non-commutative Grothendieck Inequality 
-- see the seminal work of Pisier--Shlyakhtenko \cite{PS} and Haagerup--Musat \cite{HM} -- drawn in \cite{arch} regarding Arens regularity of the algebras 
$\Ar \otg \Ar$ and $\Ar \otimes_h \Ar$, where $\Ar$ is an arbitrary $C^*$-algebra. 
In the following we collect the results from \cite{arch} which will be of importance for us. (Note that the projective operator space tensor product is of course crucial in the work \cite{arch}, 
but we shall not explicitly use it.) 
\begin{satz} \label{archivsatz} 
Let $\Ar$ and $\B$ be $C^*$-algebras. 
\begin{itemize} 
\item[(i)] If $\Ar \otimes_\gamma \B$ is Arens regular, then $\Ar \otimes_h \B$ is as well. 
\item[(ii)] If $\Ar \otimes_h \B$ and $\Ar \otimes_h \B^{op}$ are Arens regular, then so is $\Ar \otimes_\gamma \B$. 
\end{itemize} 
\end{satz} 
\begin{proof} 
(i) follows from \cite[Theorems 2.6, 2.7]{arch}, and (ii) from \cite[Theorem 2.9]{arch}. 
\end{proof} 
In the sequel, we shall use the following identification without explicit reference. 
\begin{satz}
Let $\Ar$ and $\B$ be $C^*$-algebras, and let $\R := \Ar^{**} \subseteq B(\H)$, $\Se := \B^{**} \subseteq B(\Ka)$. Then 
$$(\Ar \otimes_h \B^{op})^{**} = CB_{\R' , \Se '} (B(\Ka , \H))$$ 
via a completely isometric $w^*$-homeomorphism $\Theta$ such that $$(\Theta (a \otimes b))(T) = aTb ~\text{for all}~ a \in \Ar , b \in \B ,$$ and 
$$\Theta (X \Box Y) = \Theta(X) \Theta(Y) ~\text{for all}~ X,Y \in (\Ar \otimes_h \B)^{**}.$$ 
Moreover, by restricting $\Theta$, one obtains the completely isometric multiplicative $w^*$-homeomorphic identification 
$$\R \otimes_{eh} \Se^{op} = CB^{\sigma}_{\R' , \Se '} (B(\Ka , \H)) .$$
\end{satz}
\begin{proof} 
In view of the canonical identification $(E \otimes_h F)^{**} = E^{**} \otimes_{\sigma h} F^{**}$ for operator spaces $E$ and $F$, combined with 
\cite[Theorem 2.2]{mag}, we only need to show that $\Theta$ is multiplicative. 
To see this, let $X,Y \in (\Ar \otimes_h \B^{op})^{**}$, and let $x_i$, $y_j$ be bounded nets in $\Ar \otimes_h \B^{op}$ converging $w^*$ to $X$, respectively, $Y$. 
Note that $\Theta$ is multiplicative on $\Ar \otimes_h \B^{op}$. We then have (with limits in the $w^*$-topology): 
$$\Theta (X \Box Y) = \lim_i \lim_j \Theta(x_i y_j) = \lim_i \lim_j \Theta(x_i) \Theta (y_j) = \lim_i \Theta(x_i) \Theta(Y) =\Theta(X) \Theta(Y),$$ 
as claimed. For the second-last equation, note that every $\Theta(x_i) \in CB^\sigma(B(\Ka , \H))$, and left multiplication in $CB(B(\Ka , \H))$ by such an element is 
$w^*$-continuous; for the last equation, note that right multiplication in $CB(B(\Ka , \H))$ by any element is 
$w^*$-continuous. 
\par 
Finally, note that the restriction of $\Box$ to $\R \otimes_{eh} \Se^{op}$ defines a product on this space since $\Theta$ is multiplicative, 
$\Theta(\R \otimes_{eh} \Se^{op})=CB^{\sigma}_{\R' , \Se '} (B(\Ka , \H))$, and $CB^{\sigma}_{\R' , \Se '} (B(\Ka , \H))$ is an algebra. 
\end{proof} 

\section{The case of Varopoulos algebras} \label{comme} 
\begin{satz} \label{co-re-sc}
Let $\Ar$ be a commutative $C^*$-algebra. If $V_{\Ar}$ is Arens regular, then $\Ar$ is scattered. 
\end{satz}
\begin{proof} 
Since $\Ar$ is commutative, $V_{\Ar} = \Ag$ is isomorphic to $\Ar \otimes_h \Ar$ by the classical Grothendieck Inequality, so the latter is Arens regular. 
Represent the von Neumann algebra $\R := \Ar^{**}$ in $B(\H)$ such that it is maximal commutative; cf.\ \cite[Chapter 5, Proposition 6]{topp}. 
As $\Ar \otimes_h \Ar$ is Arens regular and commutative, $$(\Ar \otimes_h \Ar)^{**} = 
CB_{\R'}(B(\H)) = CB_{\R}(B(\H))$$ 
is commutative. Hence, we have $$CB_{\R}(B(\H)) = CB_{\R}(B(\H))^c \cap CB_{\R}(B(\H)) ,$$ where the commutant is taken in $CB(B(\H))$. Since 
$\R$ is commutative, it has no direct summand of type $I_{J,n}$ with $J$ an infinite cardinal and $n \in \mathbb{N}$. So, we have 
$CB_{\R}(B(\H))^c = CB^\sigma_{\R'}(B(\H))$, by \cite[Theorem 2.4]{mag}. Thus, as $\R = \R'$, we obtain $CB_{\R}(B(\H)) = CB^\sigma_{\R}(B(\H))$. 
By \cite[Theorem 3.5]{mag}, this implies that $\R$ is atomic. Owing to Theorem \ref{gen1}, this is equivalent to $\Ar$ being scattered. 
\end{proof} 
\begin{satz} \label{var}
Let $\Ar$ be a commutative $C^*$-algebra. The (topological) centre of ${V_\Ar}^{**}$ is isomorphic to $\Ar^{**} \otimes_{eh} \Ar^{**}$ 
(via $\varphi^{**}$ where $\varphi : \Ar \otimes_h \Ar \to \Ar \otg \Ar$ is the canonical isomorphism, based on the classical Grothendieck Inequality). 
\end{satz} 
\begin{proof} 
As above, represent $\R := \Ar^{**}$ in $B(\H)$ such that it is maximal commutative. By \cite[Theorem 2.4]{mag}, we have 
$$CB_{\R}(\B(\H))^c = CB^{\sigma}_{\R '}(\B(\H)) = CB^{\sigma}_{\R}(\B(\H)) ,$$ where the commutant is taken in $CB(B(\H))$. So we see that 
$$Z(CB_{\R}(\B(\H))) = CB_{\R}(\B(\H))^c \cap CB_{\R}(\B(\H)) = CB^{\sigma}_{\R}(\B(\H)).$$ 
Hence we obtain: 
$$\Theta (\R \otimes_{eh} \R) = CB^{\sigma}_{\R}(\B(\H)) = Z(CB_{\R}(\B(\H))) = Z (\Theta ((\Ar \otimes_h \Ar)^{**})) 
= \Theta (Z ((\Ar \otimes_h \Ar)^{**})) .$$
So, $Z ((\Ar \otimes_h \Ar)^{**}) = \Ar^{**} \otimes_{eh} \Ar^{**}$. As $V_\Ar$ is isomorphic to $\Ar \otimes_h \Ar$ (via $\varphi$), the assertion follows. 
\end{proof} 
\begin{kor} \label{irreg} 
Let $\Ar$ be a commutative $C^*$-algebra. 
The Varopoulos algebra $V_\Ar$ is strongly Arens irregular (if and) only if $\Ar$ is finite-dimensional.
\end{kor} 
\begin{proof} 
Assume that $V_\Ar$ is strongly Arens irregular. Then we have by Theorem \ref{var}: 
$$\varphi^{**} (\Ar^{**} \otimes_{eh} \Ar^{**}) = Z({V_\Ar}^{**}) = V_\Ar = \varphi^{**} (\Ar \otimes_h \Ar).$$ 
Since $\Ar \otimes_h \Ar \subseteq \Ar^{**} \otimes_h \Ar^{**} \subseteq \Ar^{**} \otimes_{eh} \Ar^{**}$, we deduce that 
$\Ar^{**} \otimes_h \Ar^{**} = \Ar^{**} \otimes_{eh} \Ar^{**}$ (as Banach spaces). By \cite[Corollary 3.8]{itoh-nagisa}, this implies that $\Ar$ is finite-dimensional. 
\end{proof} 
The result \cite[Corollary 3.8]{itoh-nagisa} to which we have referred above, implies that for a $C^*$-algebra $\Ar$, the equality $\Ar \otimes_h \Ar = \Ar \otimes_{eh} \Ar$ 
(at the Banach space level) 
holds if and only if $\Ar$ is finite-dimensional. We shall show in the following, in passing, that this equivalence fails for $\Ar = \ell_1$ 
with pointwise product. Note that the latter, endowed with any operator space structure, is completely isomorphic to an operator algebra, as shown by Blecher--Le Merdy; 
cf.\ \cite[Proposition 5.3.3]{BL1}. For our proof, recall that 
a Banach space $E$ is said to have Pelczynski's property (V) if any set $K \subseteq E^*$ satisfying 
$$\sup_{\psi \in K} | \langle \psi , x_n \rangle | \xrightarrow[~n~]{} 0$$ for every weakly unconditionally Cauchy (wuC) series $\sum x_n$ in $E$, is relatively weakly compact. 
Here, $\sum x_n$ being wuC means that $\sum | \langle \psi , x_n \rangle | < \infty$ for all $\psi \in E^*$. 
\begin{satz} 
We have (the Banach space equality) $\ell_1 \otimes_h \ell_1 = \ell_1 \otimes_{eh} \ell_1$. 
\end{satz}
\begin{proof} 
Note that, as shown by Pelczynski \cite{pe}, $c_0$ has property (V); cf.\ also \cite[p.\ 351]{pf}. 
Since $c_0^* = \ell_1$ has the Schur property, \cite[Corollary 5]{eh} implies that $c_0 \otimes_\gamma c_0$ has property (V). But by the classical Grothendieck Theorem, 
the Banach algebras $c_0 \otimes_\gamma c_0$ and $c_0 \otimes_h c_0$ are isomorphic. 
So, $c_0 \otimes_h c_0$ has property (V). Hence, its dual $\ell_1 \otimes_{eh} \ell_1$ is weakly sequentially complete, again by a well-known result of Pelczynski \cite{pe}; see also \cite[p.\ 351]{pf}. 
\par 
Thus, $\ell_1 \otimes_h \ell_1$, as a closed subspace of $\ell_1 \otimes_{eh} \ell_1$, 
is also weakly sequentially complete, so it does not contain an isomorphic copy of $c_0$. Now, by a result of Pisier \cite[Exercise 5.1]{pisier}, we have 
$$\ell_1 \otimes_h \ell_1 = (c_0 \otimes_h c_0)^* = \ell_1 \otimes_{eh} \ell_1 ,$$ 
as desired. 
\end{proof} 
\begin{bem} The following is a quick geometric proof of the well-known result that $c_0 \otimes_\gamma c_0$ is Arens regular. -- As noted above, 
since $c_0$ has property (V) and its dual has the Schur property, $c_0 \otimes_\gamma c_0$ has property (V). In view of \cite[Corollary I.2]{gi}, this implies Arens regularity. 
\end{bem} 

\section{The case of general $C^*$-algebras} \label{general}
Our goal is to prove the complete characterization given below of the Arens regularity of $\Ar \otg \Ar$. 
\begin{satz} \label{haupt}
Let $\Ar$ be a $C^*$-algebra. Then the following are equivalent: 
\begin{itemize} 
\item[(i)] $\Ar \otg \Ar$ is Arens regular; 
\item[(ii)] $\Ar$ has the Phillips property; equivalently, $\Ar$ is scattered and has the DPP. 
\end{itemize} 
In fact, if (ii) holds, then $\Ar \otimes_\gamma \B$ is Arens regular for any $C^*$-algebra $\B$. 
\end{satz} 
\begin{kor} \label{vN} 
Let $\Ar$ be a von Neumann algebra. Then the following are equivalent: 
\begin{itemize} 
\item[(i)] $\Ar \otg \Ar$ is Arens regular; 
\item[(ii)] $\Ar$ is finite-dimensional. 
\end{itemize} 
\end{kor}
\begin{proof} 
This follows from our Theorem \ref{haupt} combined with \cite[Corollary 2.14]{fu}, which states that a dual Banach space with the Phillips property is finite-dimensional.  
\end{proof} 
\begin{bem} 
As a fundamental example, we obtain that $\ell_\infty \otimes_\gamma \ell_\infty$ is Arens irregular. This solves a problem which has been open since the initial 
work of \cite{ljes} in 1981. 
\end{bem} 
We will now show that Corollary \ref{vN} does not extend to the setting of non-selfadjoint dual operator algebras, even commutative ones. Recall that a dual operator algebra 
is, by definition, a $w^*$-closed subalgebra of $B(\mathcal{H})$, for some Hilbert space $\mathcal{H}$. 
For $p \in [1, \infty )$, we shall consider $\ell_p$ with pointwise multiplication. We note that, given any Banach $\Ar$, the algebra $\ell_p \otg \Ar$ is Arens regular if and only if $\Ar$ is; 
cf.\ \cite[Corollaries 4.7 and 4.12]{u1}. 
We denote by $O \ell_p$ Pisier's operator space 
structure on $\ell_p$, obtained by complex interpolation. 
\begin{prop} 
Let $p \in [1, \infty )$. Then the algebra $\ell_p$ ($O \ell_p$) with pointwise product is (completely) isomorphic and $w^*$-homeomorphic to a commutative dual operator algebra with Arens regular 
projective tensor square. 
\end{prop} 
\begin{proof} 
First, note that $O \ell_p$ is completely isomorphic to an operator algebra, as shown by Blecher--Le Merdy; cf.\ \cite[Corollary 5.3.5]{BL1}. Since $O \ell_p$ is a dual operator space, and 
the product is separately $w^*$-continuous, there exists a commutative dual operator algebra $\Ar_p$ and a Banach algebra isomorphism and $w^*$-homeomorphism $\Phi : O \ell_p \to \Ar_p$ 
such that $\Phi$ and $\Phi^{-1}$ are completely bounded, by a result of Le Merdy's; cf.\ \cite[Theorem 5.2.16]{BL1}. Let us now prove (working in the 
Banach algebra category) that $\Ar_p \otg \Ar_p$ is Arens regular. As mentioned above, the Arens regularity of $\ell_p$ passes to $\ell_p \otg \ell_p$. 
Since $\Phi \otimes \Phi : \ell_p \otg \ell_p \to \Ar_p \otg \Ar_p$ is an algebra homomorphism, 
it suffices to show that it is surjective. As $\ell_p$ has the approximation property, we have the isometric identification 
$\ell_p \otg \ell_p = \mathcal{N} ( (\ell_p)_* , \ell_p)$, where the latter space denotes the nuclear operators 
(we write $(\ell_p)_*$ for the predual to avoid distinguishing the cases $p=1$ and $p>1$). Now let $z \in \Ar_p \otg \Ar_p$. Then we have a series representation $z = \sum x_n \otimes y_n$ with 
$\sum \| x_n \| \| y_n \| < \infty$. Hence, $z_0 := \sum \Phi^{-1} (x_n) \otimes \Phi^{-1} (y_n) \in \mathcal{N} ( (\ell_p)_* , \ell_p)$, and $(\Phi \otimes \Phi)(z_0)=z$, as desired. 
\end{proof} 
We now return to the $C^*$-case. Noting that commutative $C^*$-algebras have the DPP, our Theorem \ref{haupt} immediately yields the following. 
\begin{kor} \label{reg} 
Let $\Ar$ be a commutative $C^*$-algebra. Then the Varopoulos algebra $V_\Ar$ is Arens regular if and only if $\Ar$ is scattered. 
\end{kor} 
The following consequence of Corollaries \ref{irreg} and \ref{reg} provides a natural class of Banach algebras that are neither Arens regular nor strongly Arens irregular. 
\begin{kor} 
Let $\Ar$ be a commutative $C^*$-algebra. Then $V_\Ar$ is neither Arens regular nor strongly Arens irregular, if and only if $\Ar$ is non-scattered.
\end{kor} 
\begin{bei} \label{li} 
The case $\Ar = \ell_\infty$ provides a natural example for the above situation; note that $\ell_\infty = C(\beta \mathbb{N})$ and $\beta \mathbb{N}$ is non-scattered. By 
our Theorem \ref{var}, the (topological) centre of ${V_{\ell_\infty}}^{**}$ is isomorphic to ${\ell_\infty}^{**} \otimes_{eh} {\ell_\infty}^{**}$. 
\end{bei}
In order to prove Theorem \ref{haupt}, we start with the generalization of our Theorem \ref{co-re-sc} to arbitrary $C^*$-algebras. 
\begin{satz} \label{scatt}
Let $\Ar$ be a $C^*$-algebra. If $\Ag$ is Arens regular, then $\Ar$ is scattered. 
\end{satz}
\begin{proof} 
Since $\Ag$ is Arens regular, so is $\Ar \otimes_h \Ar$, according to Theorem \ref{archivsatz} (i). Now, 
let $\Ar_0$ be a commutative $C^*$-subalgebra of $\Ar$. Then $\Ar_0 \otimes_h \Ar_0$ is 
Arens regular as a subalgebra of $\Ar \otimes_h \Ar$. By our Theorem \ref{co-re-sc}, this entails that $\Ar_0$ is scattered. Hence, we see that 
every commutative $C^*$-subalgebra of $\Ar$ is scattered, which implies that $\Ar$ is scattered, by \cite[Lemma 2.2]{kusuda} (in fact, the last implication is an equivalence). 
\end{proof} 
We now complete the proof of the implication (i) $\Longrightarrow$ (ii) in Theorem \ref{haupt}. 
\begin{satz} 
Let $\Ar$ be a $C^*$-algebra. If $\Ag$ is Arens regular, then $\Ar$ has the DPP. 
\end{satz}
\begin{proof} 
By Theorem \ref{scatt}, $\Ar$ is scattered, whence it remains to show, 
in view of Theorems \ref{gen1} and \ref{gen2}, that the atomic von Neumann algebra $\R := \Ar^{**}$ is finite. 
Atomicity means that $\R$ is a direct sum of $B(\H_i)$, 
and we need to show that all Hilbert spaces $\H_i$ are finite-dimensional. Assume towards a contradiction that some $\H_{i_0} =: \H$ is infinite-dimensional. 
Let $\Ka$ be a Hilbert space such that $\R \subseteq B(\Ka)$. Then we have: 
$$(\Ar \otimes_h \Ar)^{**} = CB_{\R '} (B(\Ka)) \supseteq CB_{B(\H)'} (B(\Ka)).$$ 
By \cite[Theorem 3.2 and Corollary 3.3]{witt}, the latter is isomorphic to $CB_{B(\H)'}(B(\H))=CB(B(\H))$. 
\par 
As $\Ag$ is Arens regular, so is $\Ar \otimes_h \Ar$, by Theorem \ref{archivsatz} (i). Hence, the above entails that multiplication in 
$CB(B(\H))$ is separately $w^*$-continuous. Realize $\H$ as $\ell_2(\Ge)$ where $\Ge$ is a discrete group of cardinality $\text{dim}(\H)$, and note that 
$\ell_1(\Ge)^{**}$, endowed with the left Arens product, embeds isometrically and $w^*$-continuously 
as a subalgebra in $CB(B(\ell_2(\Ge)))$; cf.\ \cite[Satz 2.1.1]{neu}, \cite[Remark 3.7]{nrs}, \cite[Proposition 6.5]{hnr}. 
Indeed, the map $\Theta : \ell_\infty(\Ge)^* \to CB(B(\ell_2(\Ge)))$ given by 
$$\langle \Theta(m) T \xi , \eta \rangle = m_t ( \langle L_t T L_{t^{-1}} \xi , \eta \rangle ) ,$$ 
where $m \in \ell_\infty(\Ge)^*$, $T \in B(\ell_2(\Ge))$, and $\xi , \eta \in \ell_2(\Ge)$, yields such an embedding (here, $L_t$ denotes left translation by $t \in \Ge$). 
This entails that the (left) topological centre $Z_t(\ell_1(\Ge)^{**})$ equals $\ell_1(\Ge)^{**}$. But $Z_t(\ell_1(\Ge)^{**})=\ell_1(\Ge)$, by 
\cite[Theorem 1]{la-lo}; cf.\ also \cite[Theorem 1.1]{neu-archiv}. Thus we obtain that $\G$ is finite, a contradiction. 
\end{proof} 
The following completes the proof of Theorem \ref{haupt}. 
\begin{satz} 
Let $\Ar$ be a scattered $C^*$-algebra with the DPP. Then $\Ar \otimes_\gamma \B$ is Arens regular for any $C^*$-algebra $\B$. 
\end{satz} 
\begin{proof} 
By Theorem \ref{gen2}, $\R := \Ar^{**}$ is finite atomic. Let $\H$ be such that $\R$ is represented in standard form in $B(\H)$. 
Then $\R$ is isomorphic to $(\R')^{op}$, so it follows that $\R'$ is also finite atomic. Put $\Se := \B^{**}$. 
\par 
We first show that $\Ar \otimes_h \B$ is Arens regular. To this end, choose $\Ka$ such that $\Se^{op} \subseteq B(\Ka)$. Then we obtain: 
$$(\Ar \otimes_h \B)^{**} = CB_{\R' , (\Se^{op})'} (B(\Ka , \H)) = CB^{\sigma}_{\R' , (\Se^{op})'} (B(\Ka , \H)) ,$$ 
using in the second step \cite[Lemma 3.4]{mag}. 
Similarly, $\Ar \otimes_h \B^{op}$ is Arens regular. Indeed, choosing $\Ka$ such that $\Se \subseteq B(\Ka)$, we have: 
$$(\Ar \otimes_h \B^{op})^{**} = CB_{\R' , \Se '} (B(\Ka , \H)) = CB^{\sigma}_{\R' , \Se '} (B(\Ka , \H)) ,$$ 
using again \cite[Lemma 3.4]{mag}. 
As $\Ar \otimes_h \B$ and $\Ar \otimes_h \B^{op}$ are Arens regular,  it follows that $\Ar \otimes_\gamma \B$ is too, by Theorem \ref{archivsatz} (ii). 
\end{proof}

{} \vspace{0.7cm} 
Author's affiliations: 
\\[2ex] 
School of Mathematics and Statistics, Carleton University, 1125 Colonel By Drive, Ottawa, Ontario K1S 5B6, Canada \\ 
Email: mneufang@math.carleton.ca 
\\[1ex] 
and 
\\[1ex] 
Laboratoire de Math\'ematiques Paul Painlev\'e (UMR CNRS 8524), 
Universit\'e Lille 1 -- Sciences et Technologies, UFR de Math\'ematiques,
59655 Villeneuve d'Ascq Cedex, France \\ 
Email: matthias.neufang@math.univ-lille1.fr 

\end{document}